\documentclass[11pt]{amsart}

\usepackage{amscd}
\usepackage{bm}
\usepackage{wrapfig}

\usepackage{amsmath,amssymb,amsfonts,amsthm,bbm}
\usepackage{color,graphicx}
\RequirePackage[OT1]{fontenc}
\RequirePackage[numbers]{natbib}
\RequirePackage{hyperref}
\hypersetup{linktocpage}

\numberwithin{equation}{section}

\newcommand{\argmax}{\mathop{\rm arg\max}}

\newcommand{\R}{\mathbb{R}}

\numberwithin{equation}{section}
\newtheorem{thm}{Theorem}
\newtheorem{lem}{Lemma}
\newtheorem{rem}{Remark}
\newtheorem{prop}{Proposition}

\newcommand{\supp}{\operatorname{supp}}

\newcommand{\mF}{\mathcal{F}}

\begin{document}

\title{A regularity class for the roots of non-negative functions}

\author{Kolyan Ray}\thanks{The research leading to these results has received funding from the European Research Council under ERC Grant Agreement 320637.}
\address{Mathematical Institute, Leiden University, 2333 CA Leiden, The Netherlands}
\email{\href{mailto:k.m.ray@math.leidenuniv.nl}{k.m.ray@math.leidenuniv.nl}}
\email{\href{mailto:schmidthieberaj@math.leidenuniv.nl}{schmidthieberaj@math.leidenuniv.nl}}

\author{Johannes Schmidt-Hieber}


\begin{abstract}
We investigate the regularity of the positive roots of a non-negative function of one-variable. A modified H\"older space $\mathcal{F}^\beta$ is introduced such that if $f\in \mathcal{F}^\beta$ then $f^\alpha \in C^{\alpha \beta}$. This provides sufficient conditions to overcome the usual limitation in the square root case ($\alpha = 1/2$) for H\"older functions that $f^{1/2}$ need be no more than $C^1$ in general. We also derive bounds on the wavelet coefficients of $f^\alpha$, which provide a finer understanding of its local regularity.
\end{abstract}

\maketitle

%


%

\vspace{-1cm}

\section{Introduction}

We study the smoothness properties of roots of non-negative H\"older continuous functions on $[0,1]$. More precisely, we derive sufficient flatness-type conditions such that if $f$ is in a H\"older space with index $\beta$, then $f^\alpha$ has H\"older regularity $\alpha\beta$ for $0< \alpha\leq 1$.

This question was first studied in the square-root case ($\alpha=1/2$) by Glaeser \cite{Glaeser1963} and Dieudonn\'e \cite{Dieudonne1970}, who showed that any non-negative, twice continuously differentiable function on $\R$ admits a continuously differentiable square root. This result is sharp in general in the sense that there exist infinitely differentiable functions such that no admissible square root has H\"older index larger than one (Theorem 2.1 in \cite{Bony2006}). Flatness conditions thus become necessary for $\beta>2$ in order to ensure that a $\beta$-smooth function has an admissible square root with H\"older index greater than one. Some extensions of the result of Glaeser can be found in Lengyel \cite{Lengyel75} and Mandai \cite{Mandai1985}. Macchia \cite{Macchia78}, Reichard \cite{Reichard80}, Rainer \cite{Rainer2011} and Colombini et al. \cite{colombini2012} studied what regularity $f^{1/r}$ has in a neighbourhood of a zero of $f$, provided that enough derivatives of $f$ vanish at that point.

The key property for additional regularity of roots is flatness, in the sense that whenever $f$ is small then so are its derivatives, and in a series of recent papers, Bony et al. \cite{Bony2006, Bony2006b, Bony2010} consider flatness conditions for \textit{admissible} square roots. Recall that $g$ is an admissible square root of $f$ if $f=g^2$ (allowing $g$ to switch signs). The additional freedom to pick an admissible square root does not always help however. Suppose $f$ possesses a converging sequence of local minima $(x_n)_n$ such that $f(x_n)$ is strictly positive for all $n$ and $f(x_n)\rightarrow 0$ as $n\rightarrow \infty.$ The non-negative square root is then the most regular admissible square root and there exist infinitely differentiable functions of this form with no admissible square root having H\"older index larger than one \cite{Bony2006}. In this article we restrict to studying roots of non-negative functions that do not change sign, which we without loss of generality take to be positive.

The main result of this article is stated below, with $\|\cdot\|_{C^\beta([0,1])}$ and $\|\cdot\|_{\mF^\beta ([0,1])}$ denoting the usual and ``flat" H\"older norms on $[0,1]$ respectively, both of which are defined in Section \ref{sec.Hoeld_cones}.

\begin{thm}\label{thm.main_result}
For $\alpha \in(0,1]$, $\beta >0$ and all non-negative $f\in \mF^\beta([0,1])$,
\begin{align*}
\|f^\alpha \|_{C^{\alpha\beta}([0,1])} \leq \|f^\alpha\|_{\mF^{\alpha \beta}([0,1])} \leq C(\alpha,\beta) \|f\|_{\mF^\beta ([0,1])}^\alpha.
\end{align*}
\end{thm}

In particular we see that $f\in \mF^\beta$ is a sufficient condition for $\sqrt{f}\in C^{\beta/2}$, thereby overcoming the usual limitation that $\sqrt{f}$ need be no more than $C^1$ in general. For $\beta\leq 1$, bounds on the H\"older or Sobolev norms of $f^{1/r}$ can be found in Glaeser \cite{Glaeser1963} ($r=2$), Colombini and Lerner \cite{Colombini2003} and Ghisi and Gobbino \cite{Ghisi2013}. For bounds for the roots of more general polynomials see the recent work of Parusi{\'n}ski and Rainer \cite{Parusinski2015,Parusinski2015b,Parusinski2016}.

The condition $f\in \mF^\beta$ requires a certain notion of ``flatness" on the derivatives of $f$ that is different from that considered in \cite{Bony2006,Bony2006b,Bony2010}. In particular, it contains functions not covered by their results for which Theorem \ref{thm.main_result} nonetheless applies (see \eqref{eq.example_fcn} for a simple example). In Section \ref{sec.Hoeld_cones} we study the H\"older cone of non-negative functions satisfying such a flatness constraint and compare it to existing notions of flatness in the literature. We show that the present flatness constraints are only non-trivial for H\"older indices larger than two (Theorem \ref{thm.deriv_control}).

We also derive bounds on the wavelet coefficients of $f^\alpha$ that provide additional insight into the local regularity of $f^\alpha$ (Proposition \ref{prop.wav_decay_small}). Note that if $f\in C^\beta$ is uniformly bounded away from zero then $f^\alpha\in C^\beta$ (see Lemma \ref{lem.sqrt_lip}). This is substantially better than the regularity provided by Theorem \ref{thm.main_result}, which is driven by small function values. The wavelet bounds reflect the local function size and allow one to account for both of these regimes, thereby giving a finer account of the regularity of $f^\alpha$.

An application of this approach is in nonparametric statistics, where one aims to reconstruct a H\"older function $f$ when observing a noisy version of $\sqrt{f}$ (cf. \cite{nussbaum1996}). In the context of this problem, one must necessarily restrict to positive square roots and it is not enough to find flatness conditions ensuring that $\sqrt{f}$ has some regularity: some control of the H\"older norm is also required \cite{RaySchmidtHieber2016}.

\section{Flat H\"older smoothness}
\label{sec.Hoeld_cones}

\subsection{Definition and comparison with other flatness constraints}

Let $\lfloor \beta \rfloor$ denote the largest integer strictly smaller than $\beta.$ For $\beta>0$ and $\Omega\subset \R$, define
\begin{align*}
|f|_{C^\beta(\Omega)} = \sup_{x\neq y, x,y\in \Omega } \frac{|f^{(\lfloor \beta \rfloor)}(x) - f^{(\lfloor \beta \rfloor)}(y)| }{|x-y|^{\beta - \lfloor \beta \rfloor}}
\end{align*}
and let $C^\beta (\Omega) = \{ f: \Omega \rightarrow \R \  : \  f^{(\lfloor \beta \rfloor)} \text{ exists}, \  \| f \|_{C^\beta(\Omega)} < \infty \}$ denote the space of $\beta$-H\"older continuous functions on $\Omega$, where $\| f \|_{C^\beta(\Omega)} = \| f\|_\infty + \| f^{(\lfloor\beta\rfloor)} \|_\infty + |f|_{C^\beta (\Omega)}$. This space is often denoted by $C^{\lfloor \beta \rfloor, \beta-\lfloor\beta \rfloor}(\Omega).$ For convenience, we also write $C^\beta = C^\beta([0,1])$.

For $f\in C^\beta,$ define $|f|_{\mF^\beta}$ to be the infimum over all non-negative real numbers $\kappa$ satisfying
\begin{align}
	|f^{(j)}(x)|^\beta \leq \kappa^j |f(x)|^{\beta-j} , \ \ \forall  x\in [0,1] , \ \ \forall   \  1\leq j<\beta,
	\label{eq.flatness_cond}
\end{align} 
with $|f|_{\mF^\beta}=0$ for $\beta \leq 1$. This can be written concisely as
\begin{align*}
| f |_{\mF^\beta} = \max_{1 \leq j <\beta } \left( \sup_{x\in[0,1]} \frac{|f^{(j)}(x)|^\beta}{|f(x)|^{\beta-j}} \right)^{1/j}
= \max_{1 \leq j <\beta } \Big\| |f^{(j)}|^\beta/f^{\beta-j} \Big\|_\infty^{1/j}.
\end{align*}
The quantity $| f |_{\mF^\beta}$ measures the flatness of a function near zero. In particular, if $f$ vanishes at some point, then so do all its derivatives. Define
\begin{align*}
	\| f\|_{\mF^\beta} = \|f\|_{C^\beta([0,1])} + |f|_{\mF^\beta}
\end{align*}
and set 
\begin{align*}
	\mF^\beta = \{ f\in C^\beta \ : \  f\geq 0, \ \|f\|_{\mF^\beta}<\infty\}.
\end{align*}
This space contains for example the constant functions, any $C^\beta$-function bounded away from 0 and those of the form $f(x)=(x-x_0)^\beta g(x)$ for $g \geq \varepsilon > 0$ infinitely differentiable.

Let us compare condition \eqref{eq.flatness_cond} to the flatness conditions considered in the literature on admissible square roots. For a given function $f,$ suppose that there exists a continuous function $\gamma$, vanishing on the set of flat points of $f$, such that for any positive minimum $x_0$ of $f$, $f''(x_0) \leq \gamma (x_0) f(x_0)^{1/2}$. Theorem 3.5 of Bony et al. \cite{Bony2006} establishes that this is a necessary and sufficient condition for a four times continuously differentiable function to have a twice continuously differentiable, admissible square root. This is comparable to the flatness seminorm \eqref{eq.flatness_cond}, which for $\beta=4$ also gives that $f''$ should be bounded by $f^{1/2}.$ Nevertheless, in order to obtain bounds on the norm of $\sqrt{f},$ constraints on the third derivative of $f$ must also be imposed in our framework.

One can exploit extra regularity by assuming that $f$ and its derivatives vanish at all local minima (e.g. in \cite{Bony2006,Bony2006b,Bony2010} for admissible square roots). However, we take a different approach permitting functions to take small non-zero values. In fact, as mentioned in \cite{Bony2006}, the obstacle preventing the existence of a twice continuously differentiable square root for a general non-negative four times continuously differentiable function is a converging sequence of small non-zero minima. A sufficient condition for arbitrary even integer $\beta\geq 4$ was found in Theorem 2.2 in \cite{Bony2010}: a $\beta$-times continuously differentiable function $f$ has a $\beta/2$-times continuously differentiable square root if at each local minimum of $f$, the function and its derivatives up to order $\beta-4$ vanish. While this constraint need only hold for the local minima of $f$, forcing all local minima to be exactly zero can be overly restrictive. Indeed, a converging sequence of small non-zero minima can be allowed if it does so in a controlled way. We now give an example which is not covered by \cite{Bony2010}, but satisfies \eqref{eq.flatness_cond} and hence the conditions of Theorem \ref{thm.main_result}.

Let $K$ be a smooth non-negative function supported on $[0,1]$ and positive on $(0,1)$, for example $K(x) = \exp (-\tfrac{1}{x(1-x)})1_{(0,1)}(x)$, and set $K_\sigma (x) = K(x/\sigma)$ for $\sigma >0$. Consider the function
\begin{align}
f(x) = x^\beta- \varepsilon \sum_{j=1}^{\infty} 1_{[2^{-j-1},2^{-j}]}(x) 2^{-(j+1)\beta} K_\sigma(2^{j+1}x-1),
\label{eq.example_fcn}
\end{align}
where $0<\sigma < 2^{\beta-2}\beta^{-1} \|K'\|_\infty/\|K\|_\infty$ and $\beta \sigma 2^{1-\beta}/\|K'\|_\infty < \varepsilon \leq 1/(2\|K\|_\infty)$. Let $x\in [2^{-j-1},2^{-j}]$. Using the upper bound for $\varepsilon$, it can be checked that $f(x) \geq 2^{-1-(j+1)\beta}>0$. Using this inequality, it holds that for $k < \beta$,
\begin{align*}
|f^{(k)}(x)| \leq C(\beta,k) 2^{-j(\beta-k)} + \varepsilon 2^{-(j+1)(\beta-k)} \sigma^{-k} \|K^{(k)}\|_\infty \leq C(\beta,k,\sigma,K) f(x)^{\frac{\beta-k}{\beta}}.
\end{align*}
Since $f\in C^\beta$, it follows that $\|f\|_{\mF^\beta} \leq C(\beta,\sigma,K)<\infty$. We now show that $f$ has an infinite sequence of non-zero local minima tending to zero. Since $f'$ is continuous and $f'(2^{-j}) = \beta 2^{-j(\beta-1)}>0$ for all $j\geq 1$, it suffices to prove that every interval $[2^{-j-1},2^{-j}]$ contains a point $x_j$ with $f'(x_j) <0$. Differentiating $f$ and rearranging the desired inequality shows that a sufficient condition for the existence of such a point is $\beta < \varepsilon 2^{-(\beta-1)} \sigma^{-1} \|K'\|_\infty$, which is equivalent to the lower bound for $\varepsilon$.

In conclusion, we have shown that $f$ possesses (infinitely many) non-zero local minima, so that the conditions of Theorem 2.2 of \cite{Bony2010} are not satisfied, but that $f \in \mF^\beta$ so that Theorem \ref{thm.main_result} nonetheless applies.

\subsection{Basic properties}

In a slight abuse of notation, we say that $\|\cdot \|$ is a norm on a convex cone $A$ if $\|v\|\geq 0$ and $\|v\|=0$ if and only if $v=0,$ $\|\lambda v\| = \lambda \|v\|$ for all $\lambda>0,$ and $\|v+w\|\leq \|v\|+\|w\|$ for all $v, w\in A.$ We have thus replaced absolute homogeneity by positive homogeneity. Similarly, we say that $|\cdot|$ is a seminorm on $A$ if $|v|\geq 0$ for all $v\in A$ and both positive homogeneity and the triangle inequality hold on $A.$

\begin{thm}
\label{thm.normedspace}
The space $\mF^\beta$ is a convex cone on which $\|\cdot\|_{\mF^\beta}$ defines a norm. Moreover, $\|fg\|_{\mF^\beta} \leq C(\beta) \|f\|_{\mF^\beta} \|g\|_{\mF^\beta}$ for any $f,g\in \mF^\beta$.
\end{thm}

\begin{proof}
For the first part, it is enough to prove that $|f|_{\mF^\beta}$ is a seminorm. Observe that $|af|_{\mF^\beta} = a|f|_{\mF^\beta}$ for $a>0.$ The triangle inequality $ |f+g|_{\mF^\beta}\leq \ |f|_{\mF^\beta}+\ |g|_{\mF^\beta}$ can be restated as
\begin{align*}
	\big| f^{(j)}(x) + g^{(j)}(x) \big| \leq \big( |f|_{\mF^\beta} + |g|_{\mF^\beta}\big)^{\frac{j}{\beta}}(f(x)+g(x))^{\frac{\beta-j}{\beta}}, \quad \text{for all} \ x\in [0,1], \ j=1,\ldots, \lfloor \beta \rfloor.
\end{align*}
We may assume that $f(x),g(x),  |f|_{\mF^\beta},  |g|_{\mF^\beta}>0,$ since otherwise the result is trivial. With $r(x) :=f(x)/(f(x)+g(x))$ and $\gamma := |f|_{\mF^\beta}/ (|f|_{\mF^\beta}+ |g|_{\mF^\beta}),$ and applying Jensen's inequality to the concave function $x\mapsto x^{\frac{j}{\beta}},$
\begin{align*}
	\big| f^{(j)}(x) + g^{(j)}(x) \big|
	&\leq \ |f|_{\mF^\beta}^{\frac{j}{\beta}} f(x)^{\frac{\beta-j}{\beta}} + |g|_{\mF^\beta}^{\frac{j}{\beta}} g(x)^{\frac{\beta-j}{\beta}} \\
	&= \Big[r(x) \Big( \frac{\gamma}{r(x)}\Big)^{\frac j{\beta}}+(1-r(x))\Big(\frac{1-\gamma}{1-r(x)}\Big)^{\frac j{\beta}}\Big]
	\big( |f|_{\mF^\beta} + |g|_{\mF^\beta}\big)^{\frac{j}{\beta}}(f(x)+g(x))^{\frac{\beta-j}{\beta}} \\
	&\leq \big( |f|_{\mF^\beta} + |g|_{\mF^\beta}\big)^{\frac{j}{\beta}}(f(x)+g(x))^{\frac{\beta-j}{\beta}}.
\end{align*}

For the second statement, we establish that each term in the norm $\|fg\|_{\mF^\beta}$ is bounded by a constant times $\|f\|_{\mF^\beta} \|g\|_{\mF^\beta}$. For $\beta\in(0,1]$ the result follows immediately, so assume that $\beta >1$. Using that $x^{1/\beta} + y^{1/\beta} \leq 2^{1-1/\beta} (x+y)^{1/\beta}$ for any $x,y\geq 0$,
\begin{align*}
|(fg)^{(j)}(x)|^\beta & = \left| \sum_{r=0}^j {j \choose r} f^{(r)}(x) g^{(j-r)}(x) \right|^\beta \\
& \leq \left| \sum_{r=0}^j {j \choose r} |f|_{\mF^\beta}^{\frac{r}{\beta}} |f(x)|^\frac{\beta-r}{\beta} |g|_{\mF^\beta}^\frac{j-r}{\beta} |g(x)|^\frac{\beta-j+r}{\beta} \right|^\beta \\
& = \left( |f|_{\mF^\beta}^\frac{1}{\beta} |g(x)|^\frac{1}{\beta} + |g|_{\mF^\beta}^\frac{1}{\beta} |f(x)|^\frac{1}{\beta} \right)^{j\beta} |f(x)g(x)|^{\beta-j}  \\
& \leq \left( 2^{\beta-1} (|f|_{\mF^\beta} \|g\|_\infty + |g|_{\mF^\beta} \|f\|_\infty ) \right)^j |f(x)g(x)|^{\beta-j}  ,
\end{align*}
from which we deduce that $|fg|_{\mF^\beta} \leq 2^{\beta-1}\|f\|_{\mF^\beta} \|g\|_{\mF^\beta}$ using \eqref{eq.flatness_cond} and $\|(fg)^{(\lfloor \beta \rfloor)}\|_\infty \leq 2^{\beta-1}\|f\|_{\mF^\beta} \|g\|_{\mF^\beta}$. By the triangle inequality and arguing similarly to \eqref{Holder lemma eq3} for each term in the sum below, $|(fg)^{(\lfloor \beta \rfloor)}(x) -(fg)^{(\lfloor\beta\rfloor)}(y)|$ is bounded by 
\begin{align*}
& \sum_{r=0}^{\lfloor\beta\rfloor} {\lfloor\beta\rfloor \choose r} \left(  |f^{(r)}(x)-f^{(r)}(y)| |g^{(\lfloor\beta\rfloor-r)}(x)| + |g^{(\lfloor\beta\rfloor-r)}(x) - g^{(\lfloor\beta\rfloor-r)}(y)| |f^{(r)}(y)| \right) \\
& \quad \quad \leq C(\beta) \|f\|_{\mF^\beta} \|g\|_{\mF^\beta} |x-y|^{\beta - \lfloor\beta\rfloor},
\end{align*}
whence $|fg|_{C^\beta}\leq C(\beta)\|f\|_{\mF^\beta} \|g\|_{\mF^\beta}$. Since $\|fg\|_\infty\leq \|f\|_\infty \|g\|_\infty$, this completes the proof.
\end{proof}

\begin{rem}
The space $\mF^\beta$ inherits a notion of completeness from $C^\beta$. Suppose that $(f_n)_n \subset \mF^\beta$ is a Cauchy sequence in $C^\beta$ and $\limsup_{n\rightarrow \infty} \|f_n\|_{\mF^\beta}<\infty.$ Then the $C^\beta$-limit $f$ of $(f_n)$ satisfies $f\in \mF^\beta$ and $\|f_n\|_{\mF^\beta} \rightarrow \|f\|_{\mF^\beta}.$
\end{rem}

To see the difference between the classical H\"older space $C^\beta$ and $\mF^\beta,$ consider the functions $f_\gamma : [0,1]\rightarrow \mathbb{R},$ $f_\gamma(x) =x^\gamma,$ $\gamma>0.$ It is well-known that $f_\gamma \in C^\beta$ if and only if either $\gamma \notin \mathbb{N}$ and $\beta\leq \gamma$ or $\gamma \in \mathbb{N}$ and $\beta>0.$ The function $f_\gamma$ thus has arbitrarily large H\"older smoothness if it is a polynomial. Conversely, one can easily check that $f_\gamma \in \mF^\beta$ if and only if $\beta \leq \gamma.$ We next show that the H\"older cones $\mF^\beta$ are nested.

\begin{thm} Let $0<\beta'\leq \beta$. Then
\begin{itemize}
\item [(i)] $\mF^\beta \subset \mF^{\beta'}$ and $|f|_{\mF^{\beta'}}\leq |f|_{\mF^\beta} \vee \|f\|_{\infty},$
\item [(ii)] if $f\in \mF^\beta$ for $\beta>1$ and $\inf_x f(x) =0,$ then $|f|_{\mF^\beta}\geq \| f\|_{\infty}$ and $|f|_{\mF^{\beta'}}\leq   |f|_{\mF^\beta},$
\item [(iii)] if $f\geq c >0,$ then $f\in\mF^\beta$ if and only if $f\in C^\beta.$
\end{itemize}
\end{thm}

\begin{proof}
{\it (i):} It is enough to show that $|f|_{\mF^{\beta'}}\leq |f|_{\mF^\beta} \vee \|f\|_{\infty}.$ For $1\leq j<\beta,$ define $I_j^{\beta}:=\sup_{x\in [0,1]}(|f^{(j)}(x)|^{\beta}/f(x)^{\beta-j})^{1/j}$ and $x_j^* \in \argmax_{x\in [0,1]} |f^{(j)}(x)|^{\beta'}/f(x)^{\beta'-j}.$ Assume $f(x_j^*)>0$ since if $f(x_j^*)=0$, then the result holds trivially as $|f^{(j)}(x_j^*)|=0$. If $|f^{(j)}(x_j^*)|\leq f(x_j^*),$ then $I_j^{\beta'} \leq \|f\|_{\infty}.$ On the contrary, if $|f^{(j)}(x_j^*)|\geq f(x_j^*),$ then $I_j^{\beta'} \leq I_j^{\beta}.$ This proves $(i).$

{\it (ii):} By hypothesis, there exists $x_0\in[0,1]$ such that $f(x_0)=0$. By \eqref{eq.flatness_cond}, this implies that $f^{(j)}(x_0)=0$ for all $1\leq j<\beta,$ so that $\|f^{(j-1)}\|_\infty \leq \|f^{(j)}\|_\infty.$ This already implies $(ii),$ since for $\beta>1,$
\begin{align*}
	|f|_{\mF^\beta}
	\geq \max_{1\leq j\leq \lfloor \beta \rfloor}
	\Big(\frac{\|f^{(j)}\|_\infty^\beta}{\|f\|_\infty^{\beta-j}}\Big)^{\frac 1j}\geq \|f\|_\infty.
\end{align*}

{\it (iii)}: Under the assumptions, $|f|_{\mF^\beta} \leq \max_{1\leq j<\beta}(c^{j-\beta} \|f^{(j)}\|_\infty^\beta)^{1/j}<\infty.$
\end{proof}

Part $(i)$ of the previous theorem shows that there are two regimes. If the flatness seminorm dominates the $L^\infty$-norm, then $|\cdot|_{\mF^\beta}$ increases in $\beta,$ which matches our intuition that the seminorms should become stronger for larger $\beta.$ For functions that are very flat in the sense that $|f|_{\mF^\beta}\leq \|f\|_{\infty},$ this need not be true. Consider the function $f(x)=x+q$ on $[0,1]$ with $q>0.$ For $\beta>1,$ $|f|_{\mF^\beta}= q^{1-\beta}$ and thus for functions of low flatness the seminorm can also decrease in $\beta.$ The full flatness norms $\|\cdot\|_{\mF^\beta}$ are however unaffected by this since they also involve $\|f\|_\infty.$ Statement $(ii)$ of the previous theorem says that the low-flatness phenomenon only occurs if the function is bounded away from zero and in this case $(iii)$ shows that the flatness seminorm is always finite.

For $\beta \in(0,2]$, the additional derivative constraint is in fact always satisfied and $\mF^\beta$ contains all non-negative functions that can be extended to a $\beta$-H\"older function on $\R$.

\begin{thm}
\label{thm.deriv_control}
Suppose that $f\in C^\beta(\mathbb{R})$ for $\beta \in(0,2]$ and $f\geq 0.$ Let $f^*$ be the restriction of $f$ to $[0,1].$ Then $f^*\in \mF^\beta$ and
\begin{align*}
	|f^*|_{\mF^\beta} \leq 2^\beta |f|_{C^\beta(\mathbb{R})}.
\end{align*}
\end{thm}

\begin{proof}
The result is trivial for $\beta \leq 1$, so assume $\beta\in(1,2].$ Without loss of generality, we may assume that $|f|_{C^\beta(\mathbb{R})}=1.$ Suppose that the statement does not hold, that is $|f^*|_{\mF^\beta}> 2^\beta.$ Then there exists $x\in [0,1]$ such that $|f'(x)|>2f(x)^{(\beta-1)/\beta}.$ By symmetry it is enough to consider $f'(x)>2f(x)^{(\beta-1)/\beta}$ and we must also have $f(x)>0$, since otherwise $f'(x)=0$ as $f\in C^\beta(\R)$. Set $t= x- f(x)^{1/\beta}$ and observe that for any $y\in [t,x],$
\begin{align*}
	f'(y)\geq f'(x)- |f'(x)-f'(y)|\geq 2f(x)^{(\beta-1)/\beta} - |x-t|^{\beta-1} \geq f(x)^{(\beta-1)/\beta}.
\end{align*} 
Consequently, 
\begin{align*}
	f(t) = f(x) -\int _t^x f'(s) ds \leq f(x)  -  f(x)^{(\beta-1)/\beta} |x-t| = 0.
\end{align*}
Since $f\geq 0$ we must have $f(t)=0$ but then also $f'(t)=0\geq f(x)^{(\beta-1)/\beta} >0$, which is a contradiction. This proves the claim.
\end{proof}

Up to smoothness $\beta=2,$ flatness can thus be defined directly without the need to resort to a quantity such as \eqref{eq.flatness_cond}. This was used in the recent statistical literature, see \cite{Patschkowski2014}. To further illustrate the previous result, consider the linear function $f_1(x)=x.$ While $f_1\in C^\beta ([0,1])$ for any $\beta>0$, it has regularity one with respect to the H\"older cones. Indeed, $f_1$ cannot be extended to a non-negative function on $\R$ that is smoother than Lipschitz since the non-negativity constraint induces a kink at zero.

The previous theorem cannot be extended to smoothness indices above $\beta=2.$ Indeed, the function $f(x)=(x-1/2)^2$ is at most in $\mF^2 ([0,1]),$ but can be extended to a function in $C^\beta(\R)$ for any $\beta>0$. The reason is that for smoothness $\beta \in(1,2],$ any violation of \eqref{eq.flatness_cond} must occur at the boundary, since the smoothness and non-negativity constraints together imply \eqref{eq.flatness_cond} for interior points. For smoothness indices $\beta>2,$ the given example shows that there can be points in the interior of $[0,1]$ for which \eqref{eq.flatness_cond} does not hold.

If $f\in\mF^\beta$ then $f'$ is not necessarily in $\mF^{\beta-1}$, but the following theorem shows that integration viewed as an operator from $\mF^\beta$ to $\mF^{\beta+1}$ is bounded.
\begin{thm}
Given $f\in \mF^\beta,$ write $F:[0,1]\rightarrow [0,\infty),$ $F(x) =\int_0^x f(u) du$ for the antiderivative. Then there exists a constant $C(\beta)$ such that
\begin{align*}
	\|F\|_{\mF^{\beta+1}} \leq C(\beta) \|f\|_{\mF^\beta}.
\end{align*}
In particular, $F\in \mF^{\beta+1}.$
\end{thm}

\begin{proof}
It is sufficient to prove the result for $\|f\|_{\mF^\beta}=1.$ Observe that the H\"older seminorms agree, that is, $|F|_{C^{\beta+1}} = |f|_{C^\beta}.$ It remains to show that for some $\kappa$ which depends only on $\beta,$ $|F^{(j)}(x)|\leq \kappa^{\frac j{\beta+1}} F(x)^{\frac{\beta+1-j}{\beta+1}},$ for all $x\in [0,1]$ and all $j=0,\ldots, \lfloor \beta\rfloor +1.$ Notice that by Lemma \ref{lem.local_func_size}, $f(x+h)\geq f(x)/2$ for all $|h|\leq a f(x)^{1/\beta},$ for some positive constant $a=a(\beta).$ We firstly prove the case $j=1.$ We show that
\begin{align}
	f(x)^{\frac{\beta+1}{\beta}} \leq \tfrac 2a F(x).
	\label{eq.fx_Fx_rel}
\end{align} 
If $af(x)^{1/\beta}\leq x,$ then $f(x)^{\frac{\beta+1}{\beta}}\leq \tfrac xa f(x) \leq \tfrac 2a\int_0^x f(u) du = \tfrac 2a F(x).$ On the other hand if $af(x)^{1/\beta}\geq x,$ then $f(x)^{\frac{\beta+1}{\beta}} \leq \tfrac 2a \int_{x-af(x)^{1/\beta}}^x f(u) du \leq \tfrac 2a F(x).$ This proves \eqref{eq.fx_Fx_rel}.

Recall that $\|f\|_{\mF^\beta}=1.$ Together with \eqref{eq.fx_Fx_rel}, for $j=2, \ldots, \lfloor \beta\rfloor +1,$ $|F^{(j)}(x)| =|f^{(j-1)}(x)| \leq |f(x)|^{\frac{\beta+1-j}{\beta}}\leq (2/a)^{\frac{\beta}{\beta+1}} F(x)^{\frac{\beta+1-j}{\beta+1}}.$ This proves the differential inequalities for  $j=2, \ldots, \lfloor \beta\rfloor +1.$ With \eqref{eq.fx_Fx_rel}, the statement follows.
\end{proof}

\subsection{Wavelet bounds for $f^\alpha$}

Let $\phi$, $\psi$ be bounded, compactly supported $S$-regular scaling function and wavelet, respectively, where $S\in \mathbb{N}$; in particular we assume that $\int x^i \psi (x) dx = 0$ for $i = 0,...,S-1$ (see e.g. \cite{Hardle1998} for more details). Let $\{\phi_k: k\in \mathbb{Z}\} \cup \{\psi_{j,k} : j = 0,1,..., k \in \mathbb{Z}\}$ be the corresponding compactly supported wavelet basis of $L^2(\R)$. From this we can construct an $S$-regular wavelet basis of $L^2([0,1])$ by selecting all basis functions with support intersecting the interval $[0,1]$ and then correcting for boundary effects as explained in Theorem 4.4 of \cite{Cohen1993}. There are thus sets of indices $I_j$ with cardinality bounded by a multiple of $2^j$ such that $\{\phi_k: k\in I_{-1}\} \cup \{\psi_{j,k} : j = 0,1,..., k \in I_j\}$ forms an orthonormal basis of $L^2([0,1]).$

Recall that $f^\alpha \in C^\beta$ if $f\in C^\beta$ is bounded away from zero, rather than $f^\alpha \in C^{\alpha\beta}$ in general. We obtain two bounds on the wavelet coefficients $|\langle f^\alpha, \psi_{j,k}\rangle|$ that reflect these two regimes. The first holds for all $(j,k)$ and is most useful for those $\psi_{j,k}$ on whose support the function $f$ is small. The second bound depends explicitly on the local function values and becomes useful when $f$ is large, in which case $f^\alpha$ typically has more regularity than $\alpha\beta$ by Lemma \ref{lem.sqrt_lip}.

\begin{prop}
\label{prop.wav_decay_small}
Suppose that $\alpha\in(0,1]$, $\psi$ is $S$-regular and that $f\in\mF^\beta$ for $0 < \beta < S$. Then
\begin{equation*}
|\langle f^\alpha, \psi_{j,k} \rangle| \leq C(\psi,\alpha,\beta) \|f\|_{\mF^\beta}^\alpha \, 2^{-\frac j2 (2\alpha\beta+1)} .
\end{equation*}
For $x_0\in[0,1]$, let $j(x_0)$ be the smallest integer satisfying
\begin{equation*}
2^{j(x_0)}\geq |\supp(\psi)|a^{-1} (\|f\|_{\mF^\beta}/f(x_0))^{1/\beta},
\end{equation*}
where $a=a(\beta)$ is the constant in Lemma \ref{lem.local_func_size}. Then for any wavelet $\psi_{j,k}$ with $j\geq j(x_0)$ and $x_0\in \supp(\psi_{j,k})$,
\begin{equation*}
|\langle f^\alpha, \psi_{j,k} \rangle| \leq C(\psi,\alpha,\beta)\frac{ \|f\|_{\mF^\beta} }{f(x_0)^{1-\alpha}} 2^{-\frac{j}{2}(2\beta+1)}.
\end{equation*}
\end{prop}

The decay of the wavelet coefficients $|\langle f^\alpha, \psi_{j,k}\rangle|$ characterizes the Besov norms $\|f^\alpha\|_{B_{pq}^\beta}$ (see Chapter 4 of \cite{Gine2015} for a full definition) and one could use Proposition \ref{prop.wav_decay_small} to prove Theorem \ref{thm.main_result}. However, while $B_{\infty\infty}^\beta = C^\beta$ for non-integer $\beta$, $B_{\infty\infty}^\beta$ equals the slightly larger H\"older-Zygmund space for integer $\beta$, resulting in a slight suboptimality when $\alpha\beta \in \mathbb{N}$. While Theorem \ref{thm.main_result} provides a more concise statement, the extra local information provided by the above wavelet bounds can be crucial to obtain sharp results, for example in certain non-linear statistical inverse problems \cite{RaySchmidtHieber2016}. Recall that $f\in C^\gamma$ if and only if $|\langle f,\psi_{j,k} \rangle| \lesssim 2^{-\frac j2 (2\gamma+1)}$ for all $(j,k)$ (with a slight correction for $\gamma \in \mathbb{N}$). From Proposition \ref{prop.wav_decay_small} we can therefore conclude that on low resolution levels, that is if $j$ is small, the wavelet coefficients have the decay of a $C^{\alpha \beta}$-function. The result also shows that on high frequencies the function is $C^\beta$ and quantifies the resolution level, which depends on the function value, at which the transition between these two cases occurs.

\section{Proof of Theorem \ref{thm.main_result}}
\label{sec.proofs}

The proof is based on two technical lemmas. The  first lemma identifies a local neighbourhood of each point on which the function is relatively constant.

\begin{lem}
\label{lem.local_func_size}
Suppose that $f \in \mF^\beta$ with $\beta>0$ and let $a = a(\beta)>0$ be any constant satisfying $(e^a-1) + a^\beta / (\lfloor \beta \rfloor!) \leq 1/2.$ Then for 
\begin{align*}
	|h| \leq a \left( \frac{|f(x)|}{\|f\|_{\mF^\beta}}\right)^{1/\beta},
\end{align*}
we have
\begin{equation*}
|f(x+h) - f(x)| \leq \frac{1}{2} |f(x)| ,
\end{equation*}
implying in particular, $|f(x)|/2 \leq |f(x+h)| \leq 3|f(x)|/2$.
\end{lem}

\begin{proof}
Without loss of generality, we may assume that $\|f\|_{\mF^\beta}=1.$ By a Taylor expansion and the definition of $\mF^\beta$, there exists $\xi$ between $x$ and $x+h$ such that
\begin{align*}
|f(x+h) - f(x)| & = \left| \sum_{j=1}^{\lfloor \beta \rfloor-1} \frac{f^{(j)}(x) h^j}{j!}  +  \frac{1}{\lfloor \beta \rfloor!} f^{(\lfloor \beta \rfloor)}(\xi)h^{\lfloor \beta \rfloor} \right|  \\
& \leq \sum_{j=1}^{\lfloor \beta \rfloor} \frac{|f(x)|^{\frac{\beta-j}{\beta}} h^j}{j!}  +  \frac{1}{\lfloor \beta \rfloor!} \left| f^{(\lfloor \beta \rfloor)}(\xi) - f^{(\lfloor \beta \rfloor)}(x) \right| h^{\lfloor \beta \rfloor}   \\
& \leq \sum_{j=1}^{\lfloor \beta \rfloor} \frac{a^j }{j!} |f(x)|  + \frac{a^\beta}{\lfloor \beta \rfloor!} |f(x)| \leq \frac{1}{2}|f(x)| .
\end{align*}
\end{proof}

For $f\in \mF^\beta$, the function $f^\alpha$ satisfies a H\"older-type condition with exponent $\beta$ and locally varying H\"older constant. The following is the key technical result for establishing the smoothness of $f^\alpha$ and hence the decay of $|\langle f^\alpha, \psi_{j,k} \rangle |$. The main ingredient in the proof is a careful analysis of Fa\`a di Bruno's formula, which generalizes the chain rule to higher derivatives \cite{faadibruno1857}:
\begin{equation}
\frac{d^k}{dx^k} h(f(x)) = \sum_{(m_1,...,m_k) \in \mathcal{M}_k} \frac{k!}{m_1!...m_k!} h^{(m_1+...+m_k)}(f(x)) \prod_{j=1}^k \left( \frac{f^{(j)}(x)}{j!} \right)^{m_j} ,
\label{eq.FaaDiBruno}
\end{equation}
where $\mathcal{M}_k$ is the set of all $k$-tuples of non-negative integers satisfying $\sum_{j=1}^k j m_j = k$. Note that for $h(x) = x^\alpha$, we have $h^{(r)}(x) = C_{\alpha,r} x^{\alpha-r}$ for some $C_{\alpha,r} \neq 0$ (except the trivial case $\alpha =1$). We can relate the derivatives appearing in \eqref{eq.FaaDiBruno} to $f$ using the seminorm $|\cdot|_{\mF^\beta}$.

\begin{lem}
\label{lem.sqrt_lip}
For $\alpha\in(0,1]$, $\beta>0$, there exists a constant $C(\alpha,\beta)$ such that for all $f \in \mF^\beta$, $0\leq k <\beta$ and $x,y\in[0,1]$,
\begin{equation*}
|(f(x)^\alpha)^{( \lfloor \beta \rfloor )} - (f(y)^\alpha)^{(\lfloor \beta \rfloor)} | \leq \frac{C(\alpha,\beta) (|f|_{C^\beta} + |f|_{\mF^\beta}) }{\min (f(x)^{1-\alpha}, f(y)^{1-\alpha})} |x-y|^{\beta - \lfloor \beta \rfloor},
\end{equation*}
and 
\begin{equation}
\left| \frac{d^k}{dx^k} \left( f(x)^\alpha \right) \right| \leq C(\alpha,\beta) \|f\|_{\mF^\beta}^{k/\beta} f(x)^{\alpha-k/\beta}.
\label{eq.deriv_bds_sqrt}
\end{equation}
Moreover, if $f \geq \varepsilon > 0$,
\begin{align*}
\|f^\alpha\|_{\mF^\beta} \leq \frac{C(\alpha,\beta)}{\varepsilon^{1-\alpha}} \|f\|_{\mF^\beta}.
\end{align*}
\end{lem}

\begin{proof}
Without loss of generality assume $f(y) \leq f(x)$ and $|f|_{C^\beta} + |f|_{\mF^\beta}=1$. For $\beta \in (0,1]$, by the mean value theorem,
\begin{align*}
|f(x)^\alpha - f(y)^\alpha| \leq \sup_{f(y) \leq t \leq f(x)} \alpha t^{\alpha-1} |f(x) - f(y)| \leq \frac{|x-y|^\beta }{f(y)^{1-\alpha}}.
\end{align*}
Consider now $\beta > 1$ and write $k = \lfloor \beta \rfloor$ (the following also holds for all $k\leq\lfloor\beta\rfloor$ with certain simplifications). We must consider separately the two cases where $|x-y|$ is small and large. Let $C(\alpha,\beta)$ be a generic constant, which may change from line to line.

Suppose first that $|x-y| \leq a f(x)^{1/\beta}$ with $a$ as in Lemma \ref{lem.local_func_size}. By Lemma \ref{lem.local_func_size} we have $f(y)/2 \leq f(x) \leq 3f(y)/2$, which will be used freely without mention in the following. The proof is based on a careful analysis of Fa\`a di Bruno's formula \eqref{eq.FaaDiBruno}.

We shall establish the result by proving a H\"older bound for each of the summands in \eqref{eq.FaaDiBruno} individually. Fix a $k$-tuple $(m_1,...,m_k) \in \mathcal{M}_k$ and write $M:=\sum_{j=1}^k m_j.$ By the triangle inequality
\begin{equation}
\begin{split}
&   \left|  f(x)^{\alpha- M}  \prod_{j=1}^k \left( f^{(j)}(x) \right)^{m_j} -  f(y)^{\alpha-M} \prod_{j=1}^k \left( f^{(j)}(y) \right)^{m_j}  \right| \\
& \leq \left|  \left( f(x)^{\alpha-M} - f(y)^{\alpha - M}  \right) \prod_{j=1}^k \left( f^{(j)}(x) \right)^{m_j}  \right|  \\
& \quad \quad + \left| f(y)^{\alpha-M} \left(  \prod_{j=1}^k \left( f^{(j)}(x) \right)^{m_j} -  \prod_{j=1}^k \left( f^{(j)}(y) \right)^{m_j}  \right)  \right| .
\end{split}
\label{Holder lemma eq1}
\end{equation}
Before bounding the terms in \eqref{Holder lemma eq1}, we require some additional estimates. Firstly, by the definition of $\mF^\beta$,
\begin{equation}
\left| \prod_{j=1}^k (f^{(j)}(x))^{m_j} \right|  \leq  \prod_{j=1}^k |f(x)|^\frac{(\beta-j)m_j}{\beta}  = |f(x)|^{M - k/\beta} .
\label{Holder lemma eq2}
\end{equation}
Secondly, for any function $g$ and integer $r \geq 1$, we have by the mean value theorem $g(x)^r - g(y)^r  =  r g(\xi)^{r-1} g'(\xi) (x-y)$ for some $\xi$ between $x$ and $y$. Noting that $f(\xi) \approx f(x) \approx f(y)$, that $\beta - k \in (0,1]$ and applying the above to $g(x) = f^{(j^*)}(x)$ with $r = m_{j^*}$ and $j^* \in \{ 1,...,k-1\}$ yields
\begin{equation}
\begin{split}
& \left| \left(  f^{(j^*)} (x) \right)^{m_{j^*}} - \left(  f^{(j^*)} (y) \right)^{m_{j^*}}    \right|  \\
& \leq \left| m_{j^*}  \left(  f^{(j^*)} (\xi_{x,y,j^*}) \right)^{m_{j^*}-1}     f^{(j^*+1)} (\xi_{x,y,j^*})  (x-y) \right|^{\beta-k}\\
& \quad \times \left| \left(  f^{(j^*)} (x) \right)^{m_{j^*}} + \left(  f^{(j^*)} (y) \right)^{m_{j^*}} \right|^{1-(\beta-k)}   \\
& \leq C(\beta)  m_{j^*}^{\beta-k} |f(\xi_{x,y,j^*}) |^{\frac{(m_{j^*} (\beta-j^*)-1)(\beta-k)}{\beta}} |x-y|^{\beta-k}  |f(y)|^{\frac{m_{j^*}(\beta-j^*)(1-\beta+k)}{\beta}}  \\
& \leq C(\beta) |f(y)|^{\frac{m_{j^*}(\beta-j^*) - (\beta-k)}{\beta}} |x-y|^{\beta-k} .
\label{Holder lemma eq3}
\end{split}
\end{equation}
Strictly speaking, we cannot invoke \eqref{Holder lemma eq3} for $j^* = k$, since $f$ is only $k$-times differentiable. However, noting that $m_k$ can only take values 0 or 1, we see directly from the H\"older continuity of $f^{(k)}$ that the conclusion of \eqref{Holder lemma eq3} holds as well for $j^* =k$ since we have the bound $|x-y|^{\beta-k}$. By the same argument as \eqref{Holder lemma eq3}
\begin{equation}
\begin{split}
|f(x)^{\alpha - M} - f(y)^{\alpha - M} | & \leq  C(\alpha,\beta)  \frac{ |f(y)|^{M - 1-\alpha + k/\beta} |x-y|^{\beta-k} }{ f(x)^{ M -\alpha}  f(y)^{ M - \alpha} }  \\
& \leq C(\alpha,\beta)  |f(y)|^{-M- 1+\alpha + k/\beta} |x-y|^{\beta-k}  .
\label{Holder lemma eq4}
\end{split}
\end{equation}
Using \eqref{Holder lemma eq2} and \eqref{Holder lemma eq4}, the first term in \eqref{Holder lemma eq1} is bounded by $C(\beta) |f(y)|^{\alpha-1} |x-y|^{\beta-k}$ as required.

For the second term in \eqref{Holder lemma eq1}, we repeatedly apply the triangle inequality, each time changing the variable in a single derivative. Fix $j^*$ and define vectors $(z_j^{j^*})_{j=1}^k$, $(\tilde{z}_j^{j^*})_{j=1}^k$ that are identically equal to $x$ or $y$ in all entries and differ only in the $j^*$-coordinate, where $z_{j^*}^{j^*} = x$, $\tilde{z}_{j^*}^{j^*} = y$. Using \eqref{Holder lemma eq3}
\begin{equation}
\begin{split}
& f(y)^{\alpha - M }  \left|  \prod_{j=1}^k \left( f^{(j)} (z_j^{j^*}) \right)^{m_j}  - \prod_{j=1}^k \left( f^{(j)} (\tilde{z}_j^{j^*}) \right)^{m_j}    \right| \\
&  = f(y)^{\alpha - M }  \prod_{j=1, j\neq j^*}^k  \left( f^{(j)} (z_j^{j^*}) \right)^{m_j} \left| \left(  f^{(j^*)} (x) \right)^{m_{j^*}} - \left(  f^{(j^*)} (y) \right)^{m_{j^*}}    \right| \\
& \leq C(\alpha,\beta) f(y)^{\alpha - M} \left( \prod_{j=1,j\neq j^*}^k |f(z_j^{j^*})|^\frac{(\beta-j)m_j}{\beta} \right) |f(y)|^{\frac{m_{j^*}(\beta-j^*) - (\beta-k)}{\beta}} |x-y|^{\beta-k} \\
& \leq C(\alpha,\beta)  |f(y)|^{\alpha - M + \sum_j (\beta-j)m_j/\beta  - (\beta-k)/\beta} |x-y|^{\beta-k} \\
& = C(\alpha,\beta)  |f(y)|^{\alpha-1} |x-y|^{\beta-k} ,
\label{Holder lemma eq5}
\end{split}
\end{equation}
where in the last line we have used that $\sum_j j m_j = k$. By repeatedly applying the triangle inequality and using \eqref{Holder lemma eq5}, we can bound the second term in \eqref{Holder lemma eq1} by
\begin{equation*}
 f(y)^{\alpha - M }  \sum_{j^*=1}^k \left|  \prod_{j=1}^k \left( f^{(j)} (z_j^{j^*}) \right)^{m_j}  - \prod_{j=1}^k \left( f^{(j)} (\tilde{z}_j^{j^*}) \right)^{m_j}    \right|  \leq  C(\alpha,\beta) |f(y)|^{\alpha-1} |x-y|^{\beta-k}  ,
\end{equation*}
thereby completing the proof in the case $|x-y| \leq a f(x)^{1/\beta}$.

Applying Fa\`a di Bruno's formula \eqref{eq.FaaDiBruno} and \eqref{Holder lemma eq2} yields
\begin{equation}
\left| \frac{d^k}{dx^k} \left( f(x)^\alpha \right) \right| \leq \sum_{(m_1,...,m_k) \in \mathcal{M}_k} C f(x)^{\alpha-M} f(x)^{M - k/\beta} \leq C f(x)^{\alpha-k/\beta}.
\label{Holder lemma eq6}
\end{equation}
For $|x-y| > a f(x)^{1/\beta}$ we thus have
\begin{align*}
|(f(x)^\alpha)^{(k)} - (f(y)^\alpha)^{(k)}| & \leq C(\alpha,\beta) (f(x)^{\alpha-k/\beta} + f(y)^{\alpha-k/\beta}) \\
& \leq \frac{C(\alpha,\beta)}{f(y)^{1-\alpha}} f(x)^\frac{\beta-k}{\beta} \leq \frac{C(\alpha,\beta)}{f(y)^{1-\alpha}} |x-y|^{\beta-k},
\end{align*}
as required. This completes the proof for the first statement. Note that \eqref{eq.deriv_bds_sqrt} follows directly from \eqref{Holder lemma eq6}, since this last expression also holds for all $0\leq k<\beta$.

Suppose now $f \geq \varepsilon > 0$. It follows immediately from the results that have just been established that $\|f^\alpha\|_{C^\beta} \leq C \|f\|_{\mF^\beta}/\varepsilon^{1-\alpha}$. By \eqref{eq.deriv_bds_sqrt},
\begin{align*}
|(f(x)^\alpha)^{(j)}| \leq C\|f\|_{\mF^\beta}^{j/\beta} f(x)^{\alpha-j/\beta} (f(x)/\varepsilon)^{\frac{j(1-\alpha)}{\beta}} = C (\|f\|_{\mF^\beta}/\varepsilon^{1-\alpha})^\frac{j}{\beta} (f(x)^\alpha)^{\frac{\beta-j}{\beta}}
\end{align*}
for all $1\leq j<\beta$, so that also $|f^\alpha|_{\mF^\beta} \leq C\|f\|_{\mF^\beta}/\varepsilon^{1-\alpha}$.
\end{proof}

\begin{proof}[Proof of Theorem \ref{thm.main_result}]
By rescaling, we may assume that $\|f\|_{\mF^\beta}=1$. Throughout the proof, $C$ denotes a generic constant which only depends on $\alpha$ and $\beta$ and may change from line to line. It is enough to prove $\|f^\alpha\|_{\mF^{\alpha \beta}} \leq C.$ With \eqref{eq.deriv_bds_sqrt}, we find that $\|f^\alpha\|_\infty + \|(f^\alpha)^{(\lfloor \alpha\beta \rfloor )}\|_\infty +  | f^\alpha|_{\mF^{\alpha\beta}} \leq C.$ It thus remains to establish $|f^\alpha|_{C^{\alpha \beta}} \leq C.$

Write $\alpha \beta = \ell + \sigma$ with $\ell = \lfloor \alpha \beta \rfloor$ and $0<\sigma \leq 1.$ Let $x,y \in [0,1]$ with $x< y$ and define $\Delta_{\alpha, \ell}(x,y) := |( f(x)^\alpha)^{(\ell)} - ( f(y)^\alpha)^{(\ell)}|.$ By \eqref{eq.deriv_bds_sqrt}, 
\begin{align}
	\Big | \big( f(x)^\alpha\big)^{(\ell)} \Big| \leq C f(x)^{\alpha -\ell/\beta}.
	\label{eq.p1}
\end{align}
If $|y-x| \geq a (f(x)\vee f(y))^{1/\beta},$ we find $\Delta_{\alpha, \ell}(x,y) \leq C |x-y|^\sigma.$

On the contrary, if $|y-x| \leq a(f(x)\vee f(y))^{1/\beta},$ we conclude from Lemma \ref{lem.local_func_size} that $f(x)/2 \leq f(z) \leq 2f(x)$ for all $x\leq z \leq y.$ Consider first the case $\ell < \lfloor \beta \rfloor.$ With \eqref{eq.deriv_bds_sqrt}, 
\begin{align}
	\Delta_{\alpha, \ell}(x,y) \leq
	\sup_{z\in [x,y]} \Big|\big( f(z)^\alpha\big)^{(\ell+1)}\Big| |y-x| \leq   C f(x)^{\alpha -(\ell+1)/\beta} |y-x|.
	\label{eq.p2}
\end{align}
Write $\Delta_{\alpha, \ell}(x,y) =\Delta_{\alpha, \ell}(x,y)^{1-\sigma} \Delta_{\alpha, \ell}(x,y)^\sigma$ and bound the first factor using \eqref{eq.p1} and the second factor using \eqref{eq.p2}. This shows that $\Delta_{\alpha, \ell}(x,y)\leq C|x-y|^\sigma.$ 

Consider finally the case $\ell = \lfloor \beta \rfloor.$ Let $\theta = \sigma /(\beta -\ell)$ and observe that $\Delta_{\alpha, \ell}(x,y) =\Delta_{\alpha, \ell}(x,y)^{1-\theta} \Delta_{\alpha, \ell}(x,y)^\theta \leq C |y-x|^\sigma$, bounding the first factor using \eqref{eq.p1} and the second factor using the first assertion of Lemma \ref{lem.sqrt_lip}.
\end{proof}

\begin{proof}[Proof of Proposition \ref{prop.wav_decay_small}]
We may again assume that $\|f\|_{\mF^\beta}=1$ for both statements. If the wavelet function is $S$-regular and $f\in C^\beta([0,1])$ for $0 < \beta<S,$ then there exists a function $g$ with $\|g\|_\infty \leq 1,$ such that for any $x_0\in (0,1),$ 
\begin{align*}
	\left| \int f(x) \psi_{j,k}(x) dx \right| 
	\leq \frac{1}{\lfloor \beta \rfloor!} \left| \int \big[ f^{(\lfloor \beta \rfloor)}(x_0+g(y) (y-x_0))- f^{(\lfloor \beta \rfloor)}(x_0) \big] (y-x_0)^{\lfloor \beta \rfloor} \psi_{j,k} (y) dy \right|.
\end{align*} 
This wavelet bound can be easily proved by Taylor expanding $f$ and using the moment properties of the wavelet function. 

For $\eta >0$, note that $|f+\eta|_{C^\beta} = |f|_{C^\beta}$ and $|f+\eta|_{\mF^\beta} \leq |f|_{\mF^\beta}$, since adding a positive constant function to $f\geq0$ can only reduce the $|\cdot|_{\mF^\beta}$-seminorm. Applying Lemma \ref{lem.sqrt_lip} thus yields $|(f + \eta)^\alpha|_{C^\beta} \leq C/\eta^{1-\alpha}$. Using that $|(x + \eta)^\alpha - x^\alpha| \leq \eta^\alpha$ for all $x \geq 0$ and the wavelet bound for $C^\beta$-functions, we have
\begin{equation*}
|\langle f^\alpha, \psi_{j,k} \rangle| \leq \left| \int (f + \eta)^\alpha \psi_{j,k} \right|  +  \left|  \int ((f + \eta)^\alpha - f^\alpha) \psi_{j,k} \right| \leq \frac{C}{\eta^{1-\alpha}} 2^{-\frac{j}{2}(2\beta+1)} + C'\eta^\alpha 2^{-j/2} .
\end{equation*}
Optimizing over $\eta$ yields $\eta = 2^{-j\beta}$ and thence the first result.

We show that for all $j \geq j(x_0)$, the support of any $\psi_{j,k}$ with $\psi_{j,k}(x_0) \neq 0$ is contained in the set $\{ x : f(x_0) /2 \leq f(x) \}$. To see this observe that for any such $\psi_{j,k}$ it holds that $|\supp(\psi_{j,k})| \leq 2^{-j}|\supp(\psi)| \leq a f(x_0)^{1/\beta}$ and so applying Lemma \ref{lem.local_func_size} yields that $f(t) \geq f(x_0) - |f(t) - f(x_0)| \geq f(x_0)/2$ for any $t \in \supp(\psi_{j,k})$. Using this and applying Lemma \ref{lem.sqrt_lip}, we have that $\|f^\alpha\|_{\mF^\beta(\supp(\psi_{j,k}))} \leq C(\beta)/(f(x_0)/2)^{1-\alpha}$, where the first norm refers to $f^\alpha$ restricted to $\supp(\psi_{j,k})$ with the obvious modification of $\|\cdot\|_{\mF^\beta}$ to this set. Applying the wavelet bound above to such $(j,k)$ then yields the result.
\end{proof}

\section*{Acknowledgements} 

We would like to thank Armin Rainer for useful suggestions and an anonymous referee for suggesting a more streamlined proof.

\bibliographystyle{acm}    
\bibliography{bibhd}           

\end{document}